\documentclass[12pt]{article} 
\usepackage{latexsym}
\usepackage{amssymb}
\usepackage{amsthm}
\usepackage{graphicx} 
\usepackage{enumerate}

\newtheorem{theorem}{Theorem}[section]
\newtheorem{definition}[theorem]{Definition}
\newtheorem{lemma}[theorem]{Lemma}
\newtheorem{example}[theorem]{Example}

\def\setR{\ensuremath{\mathbb{R}}}
\def\setC{\ensuremath{\mathbb{C}}}
\def\setZ{\ensuremath{\mathbb{Z}}}
\def\setN{\ensuremath{\mathbb{N}}}

\newcommand{\delbar}{{\ensuremath{\bar\partial }}}

\def\del{{\partial}}
\def\id{\mathrm{Id}}
\def\to{\longrightarrow}
\def\N{\mathcal{N}}
\def\M{\mathcal{M}}

\def\H{\ensuremath{\mathcal{H}}} \def\J{\mathcal{J}}
  \def\F{\mathcal{F}}
\newcommand{\norm}[1]{{\ensuremath{|\!|#1|\!|}}}

\title{Compactness Results for $\mathcal{H}$--Holomorphic Maps}
\author{Jens von Bergmann}

\begin{document}
\maketitle 

{\abstract{$\H$--holomorphic maps are a parameter version of
    $J$--holomorphic maps into contact manifolds. They have arisen in
    efforts to prove the existence of higher--genus holomorphic open
    book decompositions, the existence of finite
    energy foliations and the Weinstein conjecture
    \cite{planar_weinstein}, as well as in folded holomorphic maps
    \cite{folded_holomorphic}.  For all these applications it is
    essential to understand the compactness properties of the space of
    $\H$--holomorphic maps.

    We prove that the space of $\mathcal{H}$--holomorphic maps with
    bounded periods into a manifold with stable Hamiltonian structure
    possesses a natural compactification. Limits of smooth maps are
    {\em neck-nodal maps}, i.e.  their domains can be pictured as
    nodal domains where the node is replaced by a finite cylinder that
    converges to a twisted cylinder over a closed characteristic or a
    finite length characteristic flow line. We show by examples
    that compactness fails without the condition on the periods, and
    we give topological conditions that ensure compactness.}

}

\section{Introduction}
The theory of $J$--holomorphic curves has become an indispensable tool
for symplectic and contact geometry and topology. However in many
potential applications for $J$--holomorphic curves the index for the
curves of interest turns out to be negative, or have dimension too
low.  In particular, this happens when one tries to foliate a
4--manifold, or the symplectization of a contact 3--manifold, by
$J$--holomorphic surfaces of genus $g\ge 1$. In this situation one
considers embedded curves with trivial normal bundle, so the index is
$2-2g$, when one would like the index to be 2, the dimension of the
leaf space of the foliation. To remedy this Abbas Cieliebak and Hofer
suggested using families of $J$--holomorphic curves parameterized by
$H^1(\Sigma;\setR)$ \cite{planar_weinstein}. The same parameter space
is also needed for index reasons in \cite{folded_holomorphic}. These
family version of $J$--holomorphic maps are called $\H$--holomorphic
maps.

The basic setup for $\mathcal H$-holomorphic maps goes as follows (see also
\cite{planar_weinstein} and \cite{folded_holomorphic}). Let $Z$ be a closed oriented manifold of dimension $2n+1$.
\begin{definition}\label{def:stable} $(Z,\alpha,\omega)$ is a {\em
stable Hamiltonian structure} if $\alpha\in \Omega^1(Z)$ and
$\omega\in\Omega^2(Z)$ satisfy
  \begin{eqnarray} 
    \alpha\wedge\omega^{\wedge n}>0,\qquad
    d\omega=0,\qquad \ker(\omega)\subset\ker(d\alpha)\label{eq:stable}
  \end{eqnarray}
\end{definition} 
The stable Hamiltonian structure induces as splitting $TZ=L\oplus F$,
where
\begin{eqnarray*} F=\ker(\alpha),\qquad L=\ker(d\alpha).
\end{eqnarray*} $L$ is called the {\em characteristic foliation} and
the section $R$ of $L$ defined by $\alpha(R)=1$ is called the {\em
  characteristic vector field}. $(F,\omega)$ is a symplectic vector
bundle. In the case that $d\alpha=\omega$, $\alpha$ is called a
contact form, $F$ a contact structure and $R$ the Reeb vector field.

 Let $\J$ be the set of almost complex structures on $F$ that
are compatible with $\omega$, i.e. 
\begin{eqnarray*}
  \J&=&\big\{J\in End(F)|J^2=-\id,\
  \omega(Ju,Jv)=\omega(u,v),\\ 
  &&g_F(u,v)=\omega(Ju,v)\ \mathrm{is\ a\
    Riemannian\ metric\ on\ }F\big\}.
\end{eqnarray*}
Often we will refer to a stable Hamiltonian structure as including a
choice of $J\in\J$. This gives rise to a metric
$g=\alpha\otimes\alpha\oplus g_F$ on $Z$.

\begin{definition}[$\H$--Holomorphic Maps]\label{def:H_hol} 
  Let $(\dot\Sigma,j)$ be a punctured Riemann surface. A map $v:\dot
  \Sigma\to Z$ is called {\em $\H$--holomorphic} if
  \begin{eqnarray} 
    \delbar^F_J v&=&0,\qquad\delbar_J^F=\frac{1}{2}\left(\pi_F\,dv+J\,\pi_F\,dv\,j\right)
    \label{eq:H_F}\\
    d(v^\ast\alpha\circ j)&=&0,\label{eq:H_L}\\
    \int_{\del B_p(\varepsilon)}v^\ast\alpha\circ j&=&0\qquad 
    \forall\;p\in\Sigma\setminus\dot\Sigma\quad\mathrm{and}\ \varepsilon\
    \mathrm{small\ enough}.\label{eq:H_per}
  \end{eqnarray}
\end{definition}
This is a system of elliptic differential equations, a (first order)
Cauchy-Riemann type equation in the $F$ directions and a (second
order) Poisson equation in the $L$ direction.  The last equation
demands that the periods of $v^\ast\alpha\circ j$ vanish at the
punctures. 

$\H$--holomorphic maps can be viewed as $J$--holomorphic maps with
parameter space $H^1(\Sigma;\setZ)$ in the following way. Let
$\H=\H(\Sigma,j)$ of harmonic 1--forms on $\Sigma$, i.e.
\begin{eqnarray*}
  \H&=&\{\nu\in\Omega^1(\Sigma)|\,d\nu=d(\nu\circ j)=0\}.
\end{eqnarray*}
For an $\H$--holomorphic map $v:\Sigma\to Z$ there is a unique
$\eta\in \H$, and a function $a:\dot \Sigma\to \setR$ that is unique
up to addition of a constant, so that
\begin{eqnarray}\label{eq:1-form}
  v^\ast\alpha+da\circ j=\eta\in\H.
\end{eqnarray}
Given an $\H$--holomorphic map $v$ we will
often make implicit use of the splitting of closed 1--forms given in
Equation (\ref{eq:1-form}) without mention. The pair
\begin{eqnarray*}
  \tilde v=(a,v):\dot\Sigma\to\setR\times Z
\end{eqnarray*}
is called the canonical lift of $v$ to the symplectization. It is
unique up to translation in the $\setR$--factor. With this notation
the map $v$ is $\H$--holomorphic if and only if
\begin{eqnarray*}
  \delbar_{\tilde J}\tilde v\in\H^{0,1}_\setC,
\end{eqnarray*}
where $\tilde J$ is the canonical $\setR$--invariant almost complex
structure on the symplectization and $\H^{0,1}_\setC$ is the
$(0,1)$--part of the complexification of the space of harmonic
1--forms $\H$ on $\Sigma$ viewed as taking values in the trivial
complex subbundle $\underline\setC=\tilde
v^\ast(\underline{\setR}\oplus L)\subset\tilde v^\ast T(\setR\times
Z)$. In particular, every $J$--holomorphic map is also
$\H$--holomorphic.

The important feature of the space $\H$ is that it gives a complement
of the coexact 1--forms in the space of coclosed 1--forms. Sometimes it is
convenient to choose a different complement $\tilde\H$ with some
prescribed properties and consider lifts to the symplectization
w.r.t. $\tilde \H$, i.e. maps $\tilde v=(a,v)$ with
$v^\ast\alpha+da\circ j\in\tilde \H$. 

Locally all closed forms are exact, so $\H$--holomorphic maps inherit
all local properties of $J$--holomorphic maps. By standard theory (see
e.g. \cite{hofer_asymptotics}) $\H$--holomorphic maps that satisfy
certain energy assumptions limit to closed characteristics at the
punctures and the maps extend to a continuous map from the radial
compactification $\hat\Sigma$ of $\dot\Sigma$.

In order for $\H$--holomorphic maps to be useful for applications in
symplectic and contact geometry, it is important for the moduli space
of $\H$--holomorphic maps to possess a natural compactification. The
non--compactness of this parameter space $H^1(\Sigma;\setR)$ is the
source for the compactness issues of the space of maps.

\begin{figure}[htbp]
  \centering
  \includegraphics[width=15cm]{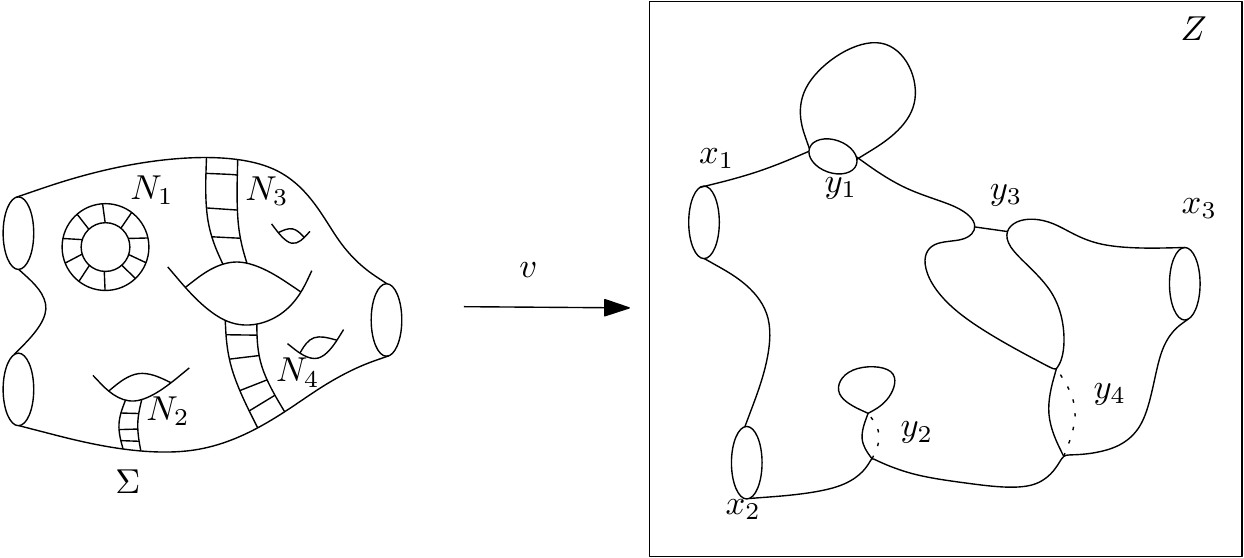}
  \caption{A map in the compactification of smooth $\H$--holomorphic
    maps. The boundary components of $\hat\Sigma$ map to closed
    characteristics $x_1$, $x_2$ and $x_3$, and the necks $N_1$, $N_2$
    and $N_4$ map to closed characteristics $y_1$, $y_2$ and
    $y_4$. The null--homologous neck $N_1$ has vanishing twist
    (``bubbles connect'') while $N_2$ and $N_4$ have in general
    non--vanishing twist. The neck $N_3$ maps to a finite length
    characteristic flow line $y_3$. 
  }
  \label{fig:general_map}
\end{figure}

The idea of using a parameter space to change the index of an equation
to fit applications has a long history, and usually requires a
delicate analysis of solutions. For example, Junho Lee
\cite{MR2060027} considered families of $J$--holomorphic maps into
K\"ahler manifolds with non--compact parameter space given by harmonic
1--forms on the target. While in that case the space of maps is in
general not compact, he was able to show that in certain interesting
cases the space of maps stay in a compact subset of the parameter
space and thus are compact.

We show that the situation for $\H$--holomorphic maps is quite
similar.  The space of maps is in general not compact (Theorem
\ref{thm:non-compactness}), and boundedness of the ``periods''
(Definition \ref{def:periods}) is a necessary and sufficient condition
on a sequence of smooth $\H$--holomorphic maps to have a convergent
subsequence (Theorem \ref{thm:compactness}, see Figure
\ref{fig:general_map}). This condition is automatically satisfied in
situations arising for important applications (Theorem
\ref{thm:compactness_S^1}). In a sequel to this we are applying these
results to nicely embedded $\H$--holomorphic maps and open book
decompositions \cite{h-hol_open_books}.

\section{Main Results}
\label{sec:compactness_statement}
In order to understand the precise compactness statement we briefly
survey some related compactness results in the literature.

\label{sec:intro} Bubble tree convergence for $J$-holomorphic maps was
established in the early '90s (\cite{gromov}, \cite{parker_wolfson},

\cite{ye}). These results build on the 
bubbling phenomenon of conformally invariant elliptic equations with
uniform energy bounds first studied by Sachs and Uhlenbeck
\cite{sachs_uhlenbeck}.

For harmonic maps, bubble tree convergence in a fixed homology class
with fixed domain complex structure was proved by Parker
\cite{parker_harmonic}. There it was also observed that a similar
result with varying complex structure on the domain does not hold due
to loss of control over the ``neck'' regions.

Chen and Tian proved \cite{chen_tian} that compactness for energy
minimizing finite energy harmonic maps with domain complex structure
converging in $\bar{\M}_g$ can be achieved if one fixes the
homotopy class of maps, rather than just homology. Then the neck maps
converge to finite length geodesics.

$J$-holomorphic maps into contact manifolds do not have uniform (or
even finite) $W^{1,2}$-energy bounds. To overcome this Hofer, Wysocki
and Zehnder (\cite{hofer_asymptotics}) tailored a suitable notion of
energy that is a homological invariant and guarantees bubble tree
convergence where nodes (and punctures) can ``open up'' to wrap closed
characteristics. In \cite{SFT_compactness} this has been extended to
the case of manifolds with stable Hamiltonian structure as targets and
allowing certain degenerations of the almost complex structure on the
target.

In the case of $\H$-holomorphic maps there are again no uniform
$W^{1,2}$-energy bounds. Roughly speaking, the $\H$-holomorphic map
equation into a $(2n+1)$- dimensional manifold with stable Hamiltonian
structure is a mixture of a $2n$-dimensional $J$-holomorphic map
equation in the almost contact planes and a 1-dimensional harmonic map
equation in the characteristic direction. This dual nature is
reflected in the compactness statement as neck maps converge to
``twisted cylinders'' over closed characteristics or finite length
characteristic flow lines.

In order to account for the possibility of necks converging to
characteristic flow lines we make the following definition for the
space of domains. 
\begin{definition}\label{def:neck-nodal_domain} Fix a genus $l$
  reference surface $\hat\Sigma$ with $m$ boundary components.  Denote
  the surface obtained by collapsing each boundary component to a
  point by $\Sigma$. A {\em neck-nodal domain} of genus $l$ with $m$
  boundary components and $k$ necks is given by a map
  \begin{eqnarray*} 
    \pi:\hat\Sigma\to C
  \end{eqnarray*} 
  to a nodal curve $C$ with $k$ nodes and $m$ marked points such that
  \begin{enumerate}
  \item each boundary component is mapped to a marked point in $C$,
    called a {\em puncture},
  \item there are $k$ embedded loops $\gamma_i$ with pairwise disjoint
    tubular neighborhoods $\nu(\gamma_i)$ bounded away from each other
    and the boundary of $\hat\Sigma$ and not containing any of the marked
    points so that each {\em neck domain}
    \begin{eqnarray*} 
      N_i=\overline{\nu(\gamma_i)}=
      \left[-\frac12,\frac12\right]\times S^1,\qquad 1\le i\le k
    \end{eqnarray*} 
    maps to a distinct node of $C$, and
  \item $\pi$ is a diffeomorphism from the smooth part of $\hat\Sigma$
    \begin{eqnarray*}
      \Sigma_0=\hat\Sigma\setminus\left(\del\Sigma\cup\mathcal{N}\right)
      \qquad\mathrm{where}\quad \mathcal{N}=\bigcup_{i=1}^m N_i
    \end{eqnarray*}
    onto the curve $C^0$ obtained from $C$ by removing the punctures
    and nodes.
  \end{enumerate} 
  Thus $\pi$ induces a complex structure on the punctured surface
  $\Sigma_0$. We denote the space of neck-nodal domains of genus $l$
  with $m$ boundary components and $n$ marked points modulo
  diffeomorphisms of $\hat\Sigma$ preserving the boundary components by
  ${\M}^N_{l,m+n}$.
\end{definition} 

\begin{figure}[htbp]
  \centering
  \includegraphics[width=11cm]{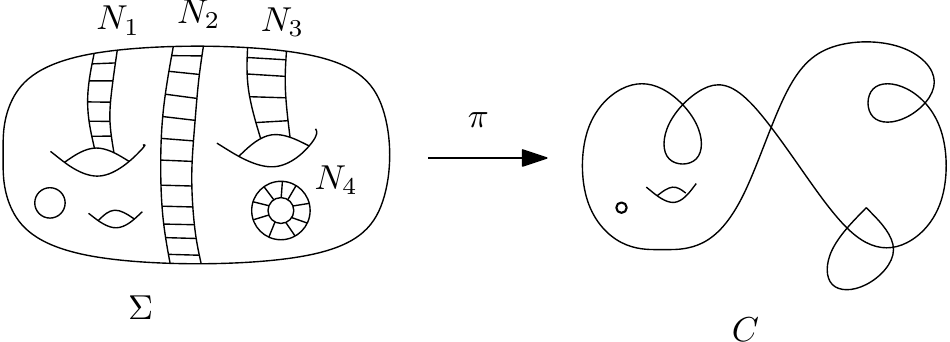}
  \caption{Neck--nodal domain.}
  \label{fig:domain}
\end{figure}

There is an obvious bijection
\begin{eqnarray*} 
  \M^N_{l,m+n}\to \M^{\$}_{l,m+n}
\end{eqnarray*} 
to the space of decorated nodal surfaces defined in
\cite{SFT_compactness}, and we endow $\M^S_{l,m+n}$ with the
same topology as $\M^{\$}_{l,m+n}$.

The neck domains $N_i$ don't carry a well defined conformal structure.
Intuitively they should be viewed as flat cylinders with
infinitesimal circumference.

\begin{definition}\label{def:neck_map} 
  A {\em neck map} $v:N=[-\frac12,\frac12]\times S^1\to Z$ is
  a map of the form
  \begin{eqnarray*} 
    v(s,t)=x(S\cdot s+T\cdot t)
  \end{eqnarray*} 
  where $x:\setR\to Z$ is a flow line of the characteristic
  vector field and $S,T\in\setR$. $T$ is called the {\em period} of
  the neck and $S$ is called the {\em twist} of the neck. 
\end{definition}
Note that if $T\ne 0$ then $x$ is necessarily a $T$-periodic orbit.

Each neck region $N\subset\hat\Sigma$ defines an element $[N]\in
H_1(\Sigma;\setZ)$, where $\Sigma$ is as always defined to be the
surface obtained from $\hat\Sigma$ by collapsing the boundary components.
\begin{definition}\label{def:minimal_twist} 
  Let $\hat\Sigma$ be a neck-nodal domain with neck domains
  $\mathcal{N}=\bigcup_{i=1}^k N_i$. Then the collection of neck maps
  $v:\mathcal{N}\to Z$ has {\em minimal twist} if whenever $\sum_{i\in
    I}[N_i]=0\in H_1(\Sigma;\setZ)$ for some index set $I$, then there
  exists a collection of non--negative real numbers $l_i$, $i\in I$
  with $\sum_{i\in I}l_i=1$ so that
  \begin{eqnarray*}
    \sum_{i\in I}l_i\,S_i=0
  \end{eqnarray*}
  where $S_i$ is the twist of $v|_{N_i}$. 
\end{definition}

We need one more definition ensuring that maps from a singular domain
can be lifted to the symplectization.
\begin{definition}\label{exact}
  A neck region $N$ is called {\em non--separating} if both boundary
  components of $N$ are adjacent to the same connected component of
  the smooth part $\Sigma_0$ of $\hat\Sigma$. 

  An $\H$--holomorphic map $v:\Sigma_0\to Z$ from the smooth
  part $\Sigma_0=\hat\Sigma\setminus(\N\cup \del\hat\Sigma)$ is called
  {\em exact} if it lifts to a map to the symplectization so that the
  $\setR$--component of the lift extends continuously by a constant
  over the non--separating components of the necks.
\end{definition}

We are now prepared for the definition of neck--nodal maps.
\begin{definition}\label{def:neck-nodal_map} 
  An $\mathcal H$-holomorphic map from a neck-nodal domain
  $(\hat\Sigma,j)$ with neck domains $\mathcal{N}=\left\{N_i\right\}_{1\le
    i\le k}$ into $Z$ is a continuous map from $\hat\Sigma$ into $Z$ that
  restricts to an exact $\mathcal H$-holomorphic map from the smooth part
  $\Sigma_0$ into $Z$, and a minimal twist neck map on the neck domains.
\end{definition} 
This definition allows for $\H$--holomorphic maps in the
compactification with qualitatively different behavior from
$J$--holomorphic maps. This is necessary as such maps occur in
examples (see the end of Section \ref{sec:circle_invariant}). If a
neck region $N$ is homologically trivial in $\Sigma$, then a minimal
twist neck map has vanishing twist on $N$. This means that
$\H$--holomorphic maps exhibit ``zero distance bubbling'', just like
in the $J$--holomorphic and harmonic map case. 

To illustrate the meaning of the minimal twist condition further let
$g$ be the genus of $\Sigma$ and denote the genus of the normalization
of corresponding nodal curve $C$ by $\tilde g$. Let $n$ be the number
of neck regions and $r$ the number of independent relations of the
neck regions in $H_1(\Sigma;\setZ)$. Then $g-\tilde g=n-r$. In this
equation $g-\tilde g$ is half the number of harmonic 1--forms that are
lost in the singular domain. Half of the lost harmonic 1--forms are
fixed as the periods of necks, and the other half is encoded in the
n twist parameters of the necks satisfying the $r$ relations.

There are several notions of energy that are important for
$\H$--holomorphic maps. With $\mathcal{A}$ the space of probability
measures on the real line we make use of the following standard
definition for $J$--holomorphic maps.
\begin{definition}\label{def:energy}
  Let $\mathcal{A}$ be the space of smooth probability measures on the
  real line.
  The {\em $\alpha$--energy $E_\alpha(v)$} and {\em $\omega$--energy
    $E_\omega(v)$} are
  \begin{eqnarray}
    \label{eq:E_alpha} E_\alpha(v)=\sup_{f\in\mathcal{A}}
    \int_{\dot\Sigma} f\circ a\,da\circ j\wedge da,\qquad
    \label{eq:E_omega} E_\omega(v)=\int_{\dot\Sigma} v^\ast\omega.
  \end{eqnarray}
\end{definition}
The integrands are pointwise non--negative functions on $\dot\Sigma$
and both energies are invariants of the relative homology class of the
map $v$. By the definition of stable Hamiltonian structure Definition
\ref{def:stable} there exists a constant $M>0$ so that
\begin{eqnarray*}
  \left|\int_{S}v^\ast d\alpha\right|\le M\int_S v^\ast \omega
\end{eqnarray*}
for any $\H$--holomorphic map $v$ from a Riemann surface, possibly
with boundary, $S$. In particular, finite $\omega$--energy of an
$\H$--holomorphic map $v:\dot\Sigma\to Z$ implies that
$\left|\int_{S}v^\ast d\alpha\right|$ is also finite for any subdomain
$S\subset\dot\Sigma$.

We will see that in order to prove compactness of a family of
$\H$--holomorphic maps it is necessary and sufficient that the
parameters stay in a compact subset of
$H^1(\Sigma;\setR)$. Compactness is to be understood with respect to
the topology induced by the period map (on a basis of
$H_1(\Sigma;\setZ)$). Here it is essential that the harmonic 1--form
$\eta$ in question is defined as the harmonic part of $v^\ast\alpha$,
and not of $v^\ast\alpha\circ j$ if one wishes to consider sequences
of complex structures converging to the boundary of
$\M_{g,m}$.

We wish to find a useful criterion to check if the periods of the
harmonic parts of $v_n^\ast\alpha$ of a sequence of $\H$--holomorphic
maps remains bounded. To this end, we associate to a canonical family
of curves along which we will evaluate the integrals of
$v^\ast\alpha$. It turns out that one-cylinder Strebel differentials
are a convenient tool for this. We quickly outline the relevant
portions of the theory, for more details we refer the interested
reader to \cite{MR743423}.

If $\Sigma$ has genus 0, then $H^1(\Sigma;\setR)$ is trivial, so every
$\H$--holomorphic map from a domain of genus 0 is automatically
$J$--holomorphic. Since the compactness properties of $J$--holomorphic
maps is already well understood we will restrict our attention to
domains of genus $\ge 1$.

A holomorphic quadratic differential is a tensor, locally in complex
coordinates $z$, given as $\phi=\phi(z)dz^2$, where $\phi(z)$ is a
holomorphic function.  $\phi$ defines a singular Euclidean metric
$|\phi(z)|\,|dz|^2$ on $\Sigma$ with finitely many singular points
corresponding to the zeros of $\phi$. $\phi$ determines a pair of
transverse measured foliations $\F_v(\phi)$ and $\F_h(\phi)$ called
the horizontal and vertical foliations given by the preimages of the
real and imaginary axes under $\phi$, respectively. Near a singular
point of $\phi$ of order $k$, $\phi$ is given in local coordinates $z$
as $z^kdz^2$. The union of the leaves both beginning and ending at a
critical point is called the critical graph $\Gamma$.

Given a non--separating simple closed curve $\gamma$ on $(\Sigma,j)$,
there exists a holomorphic quadratic differential $\phi$, called the
{\em Strebel differential}, so that its horizontal foliation has
closed leaves in the free homotopy class $[\gamma]$. Denote the set of
such Strebel differential associated with $[\gamma]$ and $j$ by
$\Phi(\gamma,j)$.

The complement $\Sigma\setminus\Gamma$ of the critical graph $\Gamma$
is a metric (w.r.t. $|\phi(z)|\,|dz|^2$) cylinder $R\subset\Sigma$. If
$\Sigma$ has genus 1, then there is no critical graph and we use one
regular leaf for $\Gamma$. For simplicity we normalize $\phi$ so that
$R=(0,1)\times S^1$ has height 1. For details see \cite{MR743423},
Theorem 21.1. Let $\{\sigma_s(\phi):S^1\to R\}_{s\in(0,1)}$ be a
parametrization the closed leaves of the horizontal foliation
$\F_h(\phi)$.

Armed with this definition we are ready to define the {\em periods} of
$\H$--holomorphic maps.
\begin{definition}\label{def:periods}
  Let $\gamma$ be a non--separating simple closed curve in $\Sigma$
  and $\mu$ a 1--form on $\Sigma$. Then the {\em period of $\mu$ along
    $\gamma$} are
  \begin{eqnarray*}
    P_{[\gamma]}(\mu)=\sup_{\phi\in\Phi(j,[\gamma])}\sup_{s\in(0,1)}
      \left|\int_{\sigma_s(\phi)}\mu\right|.
  \end{eqnarray*}
  
  For an $\H$--holomorphic map $v:\dot\Sigma\to Z$ define the period
  \begin{eqnarray*}
    P_{[\gamma]}(v)=P_{[\gamma]}(v^\ast\alpha).
  \end{eqnarray*}
  
  We say that a family of maps $(v_n,j_n)$ has {\em bounded periods} if
  there exists a collection of simple closed curves $\gamma_i$ that
  form a basis of $H_1(\Sigma,\setZ)$ so that the associated periods
  $P_{[\gamma_i]}(v_n)$ are uniformly bounded.
\end{definition}

The definition of the periods along a curve $\gamma$ is somewhat
abstract. Intuitively, we think of the periods as the periods of the
non--closed form $v^\ast\alpha$. The one--cylinder Strebel
differentials allow us to define the periods in a way that is
invariant under the gauge action by diffeomorphisms, and independent
of the choice of conformal metric on the domain. It turns out that the
periods are essentially given by the periods of the harmonic part of
the co--closed form $v^\ast\alpha$ (see Lemma \ref{lem:eta_periods}).

Bounded periods are a necessary condition for any meaningful
compactness result, as the a sequence of maps with unbounded periods
has unbounded diameter in the image. The following theorem shows that
the converse is also true, i.e. that bounded periods lead to
compactness.

\begin{theorem}\label{thm:compactness}
  Let $(Z,\alpha,\omega,J)$ be a stable Hamiltonian structure so that
  all periodic orbits are Morse or Morse-Bott. The space of smooth
  $\H$-holomorphic maps into $Z$ with uniformly bounded $\omega$ and
  $\alpha$--energies with uniformly bounded periods has compact
  closure in the space of neck--nodal $\H$--holomorphic maps.
\end{theorem}
In Section \ref{sec:circle_invariant} we give examples of topological
conditions that guarantee that the periods of families of maps are
uniformly bounded, leading to compact moduli spaces. Another such
condition guaranteeing bounded periods is given in a follow--up paper
\cite{h-hol_open_books} when considering nicely embedded maps.

The condition on the periods is not vacuous as the following result
shows.
\begin{theorem}\label{thm:non-compactness} 
  Let $(\dot T,i)$ be the twice--punctured standard torus and let
  $S^3$ be equipped with the standard contact form and complex
  structure. There exists a smooth family $v_t$ of $\H$--holomorphic
  maps parametrized by $\setR$ so that a sequence $v_{t_n}$,
  $t_n\in\setR$ has a convergent subsequence if and only if $t_n$ has
  a convergent subsequence in $\setR$ as $n\to\infty$. The
  width of the image becomes unbounded as $t_n\to\infty$.
\end{theorem}
The existence of non--compact smooth families of maps stands in stark
contrast to the case of $J$--holomorphic maps and destroys any hope
for a general compactness theorem for $\H$--holomorphic maps.

\section{Bounded Periods and Compactness}\label{sec:energy} 
In this section we prove Theorem \ref{thm:compactness}. In the first
part, we show that the requirement of bounded periods of the maps
leads to bounded periods of the harmonic 1--forms, and that every
sequence with bounded periods possesses a subsequence that converges
on compact subsets of the complement of the necks.

In the following subsection we investigate the convergence of the maps
on the necks, where the requirement of bounded periods will lead to
neck maps with minimal twist.

Putting these results together we then proceed to prove Theorem
\ref{thm:compactness}.

First we explain the metrics we are using on the domains.  Following
the construction in Section 4 of \cite{IP_sum} we choose a family of
metrics on the space of the domains $(\dot\Sigma,j)$,
$j\in\overline\M_{g,n}$ of $\H$--holomorphic maps that comes from a
metric on the universal curve on the thick part of the domain and is
given by the cylindrical metric on (a cofinite subset of) the thin
part. More specifically we adopt the definition of the weight function
$\rho:(\dot\Sigma,j)\to \setR$ and work with the metric
$g=\rho^{-2}\cdot \tilde g$, where $\tilde g$ is the restriction of a
Riemannian metric on the universal curve. In particular, near a
puncture we have local coordinates $C_0=[0,\infty)\times
S^1\subset\dot\Sigma$ with metric $g=ds^2+dt^2$, $s\in [0,\infty)$ and
$t\in S^1$, and $\rho^2(s,t)=8e^{-2s}$, and on a neck cylinder
$C_R=[-R,R]\times S^1\subset\dot\Sigma$ we have the same flat metric
and $\rho^2(s,t)=8e^{-2R}\cosh(2s)$. Given a sequence of maps and
conformal structures $(v_m,j_m)$ we may adjust the space
$\overline\M_{g,n}$ every time we rescale (a finite number of times)
by adding marked points as needed, and we will adjust the metrics
accordingly without making explicit mention of this.

We refer to this metric as the {\em cylindrical metric} and will use
it throughout for estimates. For the final statement of the
compactness theorem we will however use a different metric, namely the
non--conformal metric where the cylinders of the thin part are rescaled
along the height of the cylinders to $[-1,1]\times S^1$, where the
scaling function depends on the asymptotic approach of the maps to a
closed characteristic. This metric extends to a smooth metric on the
space of neck--nodal domains, and the convergence results in the
compactness statement are to be understood with respect to this
metric.

\begin{lemma}\label{lem:da_periods}
  Let $(v,j)$ be a an $\H$--holomorphic map and let
  $E=E_{d\alpha}(v)+|T|$ be the sum of the $d\alpha$ energy of $v$ and
  $T$ be the of the sum of the absolute values of the periods of $v$
  at the punctures. Then, for any non--trivial free
  homotopy-separating simple closed curve $\gamma$, the periods
  $P_{[\gamma]}(da\circ j)\le E$, where $-da\circ j$ is the co-exact
  part of $v^\ast\alpha$, i.e. $v^\ast\alpha+da\circ j\in\H$.
\end{lemma}
\begin{proof}
  Let $\sigma_s$ be a foliation with compact leaves given by a
  Strebel differential $\phi$ with ring domain $R$ of height 1 associated
  with $\gamma$, and let $\Gamma$ denote it's singular leaf. So
  $\sigma_s=\{s\}\times S^1\subset R=(0,1)\times S^1$.
  \begin{eqnarray*}
    \int_0^1\left(\int_{\sigma_s}da\circ j\right)ds=\int_Rda\circ j\wedge ds=\int_R
    da\wedge dt=\int_{\del R}a\,dt=\int_{\Gamma}a\,dt-\int_{\Gamma}a\,dt=0.
  \end{eqnarray*}
  Here we interpret integrals on a leaf containing a puncture in the
  sense of Cauchy.  Then for any $s,\tilde s\in(0,1)$
  \begin{eqnarray*}
    \left|\int_{\sigma_s}da\circ j-\int_{\sigma_{\tilde s}}da\circ j\right|\le E_{d\alpha}(v)+|T|
  \end{eqnarray*}
  where $|T|$ is the sum of the absolute value of the periods, so
  \begin{eqnarray}\label{eq:coexact_integral}
    \left|\int_{\sigma_s}da\circ j\right|
    \le \left|\int_R da\circ j\wedge ds\right|+E_{d\alpha}(v)+|T|
    \le E_{d\alpha}(v)+|T|.
  \end{eqnarray}

  Since this is true for any Strebel differential $\phi$ and any leaf
  $\sigma_s$ we conclude that
  \begin{eqnarray*}
    P_{[\gamma]}(da\circ j)=\sup_{\phi\in\Phi([\gamma],j)}\sup_{s\in(0,1)}
    \left|\int_{\sigma_s}da\circ  j\right|\le E.
  \end{eqnarray*}
\end{proof}

We immediately obtain the following results.
\begin{lemma}\label{lem:eta_periods}
  Let $(v_n,j_n)$ be a sequence of $\H$--holomorphic maps with fixed
  asymptotics and uniformly bounded $\omega$--energy.  Then $v_n$ has
  uniformly bounded periods if and only if the periods of
  $\eta_n=v_n^\ast\alpha+da_n\circ j_n\in\H$ are uniformly bounded.
\end{lemma}

\begin{lemma}\label{lem:eta_bounded_rescale}
  Let $(v_n,j_n)$ be a sequence of $\H$--holomorphic maps so that the
  periods of $\eta_n=v_n^\ast\alpha+da_n\circ j_n\in \H$ are uniformly
  bounded. Then, after finitely many rescalings, there is a
  subsequence so that $dv_n$ is uniformly bounded.
\end{lemma}
\begin{proof}
  By Lemma \ref{lem:harmonic_bounded} we see that $\eta_n$ is uniformly
  bounded. The result now follows from the standard bubbling off
  analysis.
\end{proof}

\begin{definition}\label{def:twist}
  For an $\H$--holomorphic map $v:C_R\to Z$ from a cylinder
  $C_R=[-R,R]\times S^1$ we define the {\em twist}
  \begin{eqnarray*}
    S=\int_{S^1}\left(\int_{[-R,R]\times \{t\}}v^\ast\alpha\right)\,dt
  \end{eqnarray*}
  and the {\em average twist}
  \begin{eqnarray*}
    \tilde S=\frac{S}{2R}=\frac1{2R}
    \int_{S^1}\left(\int_{[-R,R]\times \{t\}}v^\ast\alpha\right)\,dt.
  \end{eqnarray*}
\end{definition}
The twist and the average twist only depend on $\eta$ and are
independent of $da$, since
\begin{eqnarray}
  \int_{S^1}\left(\int_{[-R,R]\times \{t\}}da\circ j\right)\,dt
  =\int_{[-R,R]\times S^1}a_t\,ds\,dt
 =\int_{[-R,R]\times \{t\}}\left(\int_{S^1}da\right)\,ds=0.
\end{eqnarray}
In particular, the twist of a neck--region is uniformly bounded in
terms of the periods of $\eta$ by Lemma \ref{lem:bounded_twist}, and
the relative twist is bounded in terms of $\norm{\eta}_\infty$.

\begin{theorem}\label{thm:bounded_gradient}
  Let $v_n:(\dot\Sigma,j_n)\to Z$ be a sequence of $\H$--holomorphic
  maps with $j_n\in\M_{g,n}$ and uniformly bounded $E_\alpha$--energy
  and $E_\omega$--energy and uniformly bounded periods.

  Then there exists a constant $C>0$ and a subsequence so that with
  $v_n^\ast\alpha=\eta_n-da_n\circ j_n$
  \begin{eqnarray*}
    \norm{\eta_n}_\infty<C,\qquad
    \norm{da_n}_\infty<C,\qquad\mathrm{and}\quad
    \norm{\pi_F\,dv_n}_\infty<C
  \end{eqnarray*}
  and the twist of all neck--maps is uniformly bounded.
\end{theorem}
\begin{proof}
  By Lemma \ref{lem:eta_periods} we see that the periods of $\eta_n$
  are uniformly bounded. Then Lemma \ref{lem:eta_bounded_rescale}
  shows that we can, after finitely many rescalings, choose a
  subsequence so that $dv_n$ is uniformly bounded. By Lemma
  \ref{lem:harmonic_bounded} we see that the sup norm of $\eta_n$ is
  also uniformly bounded, and thus $da_n$ and $\pi_F\,dv_n$ must also
  be uniformly bounded in the sup norm.  The twists of the neck maps
  are uniformly bounded by Lemma \ref{lem:bounded_twist}.
\end{proof}

\subsection{Long Cylinders}
\label{sec:long_cylinders}
To prove the compactness statement we need to understand the behavior
of long $\H$--holomorphic cylinders with small $\omega$--energy and
uniformly bounded derivative, center action, and twist. We reduce the
argument to the $J$--holomorphic case discussed in
\cite{hofer_cylinders}. The main difference between the $J$ and
$\H$--holomorphic settings is that $\H$--holomorphic maps may have
non--zero (uniformly bounded) twist, whereas $J$--holomorphic maps
have vanishing twist.

To reduce the question about $\H$--holomorphic cylinders to
$J$--holomorphic cylinders, let $\phi_t:Z\to Z$ denote the
time--$t$ characteristic flow. The bundle map $d\phi_t:TZ\to
\phi_t^\ast TZ$ is an isomorphism preserving the splitting $TZ=L\oplus
F$. Given an $\H$--holomorphic map $v:C_R\to Z$ with average
twist $\tilde S$ let $\tilde v:C_R\to Z$ be given by
\begin{eqnarray*}
  \tilde v(s,t)=\phi_{-\tilde S\,s} v(s,t)
\end{eqnarray*}
and define the 1--parameter family of almost complex structures
\begin{eqnarray*}
  J_s\in End(F),\qquad
  J_s(z)=(\phi_{\tilde S\,s}^\ast J)(z)
  =d\phi_{-\tilde S\,s}(\phi_{\tilde S\,s}(z))\circ J(\phi_{\tilde
    S\,s}(z))\circ d\phi_{\tilde S\,s}(z).
\end{eqnarray*}
Let $\tilde J=\tilde J(s,t,z):C_R\times Z\to End(F)$ be the
domain dependent almost complex structure defined by $\tilde
J(s,t,z)=J_s(z)$. Then $\tilde v$ is $\tilde J$--holomorphic, that is
$\tilde v^\ast\alpha=v^\ast\alpha-\tilde S\,ds$ is coexact and
\begin{eqnarray*}
  \delbar^F_{\tilde J}\tilde v(s,t)&=&
  \frac12\left\{\pi_F\,d\tilde v(s,t)+J_s(\tilde v(s,t))\circ \pi_F\,d\tilde
  v(s,t)\circ j\right\}\\
  &=&\frac12d\phi_{-\tilde S\,s}(v(s,t))
  \left\{\pi_F\,dv(s,t)+J(v(s,t))\circ \pi_F\,dv(s,t)\circ j\right\}\\
  &=&0.
\end{eqnarray*}

Now suppose the twists $S_n$ of a family of $\H$--holomorphic maps are uniformly
bounded by $|S_n|<C$. Then for each $z\in Z$, the family $J_{s}(z)$, of
almost complex structures varies in the compact set
\begin{eqnarray*}
  \mathcal{J}(z,C)
  =\left\{d\phi_{-\sigma}(\phi_{\sigma}(z))\circ J(\phi_{\sigma}(z))\circ 
  d\phi_{\sigma}(z)\,|\,\sigma\in\left[-C/2,C/2\right]\right\}.
\end{eqnarray*}
independent of how large $R$ is.

We note that all the results in \cite{hofer_cylinders} remain valid
when the fixed almost complex structure in \cite{hofer_cylinders} is replaced
by a domain--dependent almost complex structure varying in a compact
set.

Before we proceed we need the following definition.
\begin{definition}\label{def:period_spectrum}
  The {\em period spectrum} of $(Z,\alpha,\omega)$ is
  \begin{eqnarray*}
    \mathcal{P}=\{0\}\cup\{T> 0|\,\exists\ \mathrm{closed\ characteristic}\ x\,
    \mathrm{of\ period\ }T\}.
  \end{eqnarray*}

  
  For any $E>0$ the {\em period gap} w.r.t $E$ is the largest number
  $\hbar=\hbar(E)$ so that
  \begin{eqnarray}\label{eq:hbar}
    |T-T'|<\hbar\quad\forall\, T,T'\in\mathcal{P}\ \mathrm{with}\ T,T'< E.
  \end{eqnarray}
\end{definition}

\begin{lemma}\label{lem:convergence_neck}
  Let $E_0>0$ so that all closed characteristics of period $T\le E_0$
  are non--degenerate. Let $\hbar$ be the period gap between closed
  characteristics of period $\le E_0$ as in Equation
  (\ref{eq:hbar}). Let $1>\delta>0$ be smaller than the lowest
  eigenvalue of any asymptotic operator governing the transverse
  approach to any closed characteristic of period $T\le E_0$. 

  Fix $\gamma$ satisfying $0<\gamma<\hbar\le E_0$ and $N\in
  \setN$. Then for every $\varepsilon>0$ there exists a constant $h>0$
  so that the following holds.

  For every $R>h$ and every $\H$--holomorphic cylinder
  \begin{eqnarray*}
    v:C_R=[-R,R]\times S^1\to Z
  \end{eqnarray*}
  satisfying $E_\omega(v)\le \gamma$ and gradient and twist bounded by
  $C$ and center action $T=\int_{\{0\}\times S^1}\eta$ satisfies $T\le
    E_0-\gamma$ there exist a characteristic flow line $x$ so that
  \begin{eqnarray*}
    d(v(s,t),x(\tilde S\,s+T\,t))
    &\le& \varepsilon e^{-\delta(R-h)}\cosh(\delta s),\qquad
    \forall (s,t)\in C_{R-h}\\
    |D^\nu(dv(s,t)-\tilde S\,ds-T\,dt)|
    &\le& \varepsilon e^{-\delta(R-h)}\cosh(\delta s),\qquad
    \forall (s,t)\in C_{R-h},\quad \forall \nu,\ |\nu|\le N.
  \end{eqnarray*}
\end{lemma}

\begin{proof}
  Consider the $\tilde J=\tilde J(\tilde S)$--holomorphic map
  \begin{eqnarray*}
    \tilde v(s,t)=\phi_{-\tilde S\,s}v(s,t)
  \end{eqnarray*}
  and note that $E_{\omega}(\tilde v)=E_\omega(v)$ and $\tilde v$
  lifts to a finite $\alpha$--energy $\tilde J$--holomorphic map into the
  symplectization. Moreover the $\alpha$ energy is a--priori bounded
  in terms of the center action and the $\omega$--energy.

  To prove the theorem we need to show that there exists $h>0$ so that
  \begin{eqnarray*}
    d(\tilde v(s,t),x(T\,t))
    &\le& \varepsilon e^{-\delta(R-h)}\cosh(\delta s),\qquad
    \forall (s,t)\in C_{R-h}\\
    |D^\nu(d\tilde v(s,t)-T\,dt)|
    &\le& \varepsilon e^{-\delta(R-h)}\cosh(\delta s),\qquad
    \forall (s,t)\in C_{R-h},\quad \forall \nu,\ |\nu|\le N.
  \end{eqnarray*}
  But this follows directly from Theorems 1.2 and 1.3 of
  \cite{hofer_cylinders}.
\end{proof}

\subsection{Proof of Theorem \ref{thm:compactness}}
\label{sec:compactness_results}

Before we proof Theorem \ref{thm:compactness} we observe some
relations among the neck lengths.
\begin{lemma}\label{lem:minimal_twist}
  Let $j_n\to j_0\in \M_{g,n}$ be a sequence of complex structures on
  $\Sigma$ and $\eta_n$ a sequence of harmonic 1--forms with
  converging period integrals. Then the twists $S_i^n$ of $\eta_n$ on
  each neck $N_i$ converge to real numbers $S_i$, and there exists a
  subsequence so that whenever $\sum_{i\in I}[N_i]=0\in
  H_1(\Sigma;\setZ)$ for some index set $I$ there exist non--negative
  real numbers $l_i$ with $\sum_{i\in I}l_i=1$ so that $\sum_{i\in
    I}l_i\cdot S_i=0$.

  In particular, homologically trivial necks have vanishing twist.
\end{lemma}
\begin{proof}
  Consider $\eta_n$ on the necks $N_i=[-R_i^n,R_i^n]\times S^1$.  With
  $N_i^0=\{0\}\times S^1$ the center loops of each neck. The periods
  of $\eta_n\circ j$ on $N_i^0$ satisfy
  \begin{eqnarray*}
    \tilde S_i^n=\int_{N_i^0} \eta_n\circ j.
  \end{eqnarray*}
  Let $I$ be an index set so that $\sum_{i\in I} [N_i]=0\in
  H_1(\Sigma;\setZ)$ and define, for $j\in I$
  \begin{eqnarray*}
    l_j^n=\frac{R_I^n}{R_j^n}\in (0,1],\qquad R_I^n=\left(\sum_{i\in I}\frac{1}{R_i^n}\right).
  \end{eqnarray*}
  The twist $S_i^n$ of $\eta_n$ on the neck $=N_i$ satisfy,
  \begin{eqnarray*}
    \sum_{i\in I}l_i^n \cdot S_i^n
    =2R_I^n\sum_{j\in I}\tilde S_j^n
    =0.
  \end{eqnarray*}
  Set $S_i=\lim_{n\to \infty}S_i^n$, which exists by
  assumption, and choose a subsequence so that
  $l_i=\lim_{n\to\infty}l_i^n\in[0,1]$ exists. For each $n$ we have
  \begin{eqnarray*}
    \sum_{i\in I}l_i^n=\sum_{i\in I}\frac{R_I^n}{R_j^n}=1
  \end{eqnarray*}
  so
  \begin{eqnarray*}
    \sum_{i\in I}l_i=1,\qquad\mathrm{and}\ \sum_{i\in I} l_i\cdot S_i=0.
  \end{eqnarray*}

  Since there are only finitely many index sets $I$ so
  that $\sum_{i\in I} [N_i]=0$ there exists a subsequence so that the
  Lemma holds true.
\end{proof}

\begin{proof}[Proof of Theorem \ref{thm:compactness}]
  Let $v_n:\dot\Sigma\to Z$ be a sequence of smooth $\H$--holomorphic
  maps with bounded $\omega$ and $\alpha$--energies and periods
  bounded by $C$. We need to show that there exists a subsequence that
  converges to a neck--nodal $\H$--holomorphic map.

  By Theorem \ref{thm:bounded_gradient} we may pass to a subsequence
  so that $j_n\to j_0$ and $|dv_n|$, $|\eta_n|$ and the
  relative twists of the neck maps are uniformly bounded. By elliptic
  regularity and Arzela--Ascoli we extract a convergent (in
  $C^\infty$) on the thick part of $(\Sigma,j_0)$. 

  By Lemma \ref{lem:minimal_twist} we may extract a subsequence so
  that the twists of $v_n$ on the necks are bounded, convergent, and
  the twists are minimal.

  By our assumption we have that the center action of each neck is
  bounded by some constant $E\le E_0$, so by Lemma
  \ref{lem:convergence_neck} we see that there exists constants
  $C,h,\delta>0$ and a characteristic flow line $x$ so that
  $d(v_n(s,t)-x(s,t))<Ce^{-\delta(R_n-h)}\cosh(\delta s)$ and
  $|dv_n(s,t)-(\tilde S\,ds+T\,dt)\otimes R|<Ce^{-\delta(R_n-h)}\cosh(\delta s)$ for
  $(s,t)\in C_{R_n-h}=[-R_n+h,R_n-h]\times S^1$. Since the gradient of
  $v_n$ is uniformly bounded on $C_{R_n}$ we may assume, by adjusting
  the constant $C$ in the above formulas, that $h=0$.

  Set $\mu=\delta/2$.  We may split a neck region
  $C_{R_n}=[-R_n,R_n]\times S^1$ into regions
  \begin{eqnarray*}
    A_n^-&=&[-R_n,-R_n+\ln(R_n)]\times S^1\\
    B_N&=&[-R_n+\ln(R_n),R_n-\ln(R_n)]\times S^1\\
    A^+_n&=&   [R_n-\ln(R_n),R_n]\times S^1.
  \end{eqnarray*}
  We similarly split up the cylinder $C_1=[-1,1]\times S^1$ into
  regions
  \begin{eqnarray*}
    \tilde A^+=[-1,-\frac12]\times S^1,\quad
    N=[-1/2,1/2]\times S^1,\quad
    \tilde A^+=[\frac12,1]\times S^1
  \end{eqnarray*}
  and define the piecewise diffeomorphism $\phi_n:C_1\to C_{R_n}$ via
  the diffeomorphisms
  \begin{eqnarray*}
    \phi^-_n&:&\tilde A^-\to A_n^-,\qquad
    \phi_n^-(r,t)=(-R_n-\mu^{-1}\ln[1-2(1-R_n^{-\mu})(1+r)],t)\\
    \phi_n^N&:&N\to B_n,\qquad
    \phi_n^N(r,t)=(2(R_n-\ln(R_n))r,t)\\
    \phi^+_n&:&\tilde A^+\to A_n^+,\qquad
    \phi_n^+(r,t)=(R_n+\mu^{-1}\ln[1-2(1-R_n^{-\mu})(1-r)],t).
  \end{eqnarray*}

  Define the map $\tilde v_n=v_n\circ \phi_n:C_1\to Z$ and consider
  the restrictions to the subdomains $\tilde v_n^\pm=\tilde
  v_n\big|_{\tilde A^\pm}$ and $\tilde v_n^N=\tilde v_n\big|_{N}$.
  Then, with $(r,t)$ coordinates on $\tilde A$ and remembering that
  $\delta/\mu=2$
  \begin{eqnarray*}
    |(d\tilde v_n^\pm(r,t)-T\,dt)(\del_t)|
    &\le& C e^{-\delta R_n}
    \cosh(\pm\delta R_n\pm 2\ln[1-2(1-R_n^{-\mu})(1\mp r)])\\
    &\le& C[ 1-2(1-R_n^{-\mu})(1\mp r)]^2
  \end{eqnarray*}
  which is uniformly bounded by $C$ and converges to zero as $r\to \pm\frac12$
  and $n\to \infty$. Similarly
  \begin{eqnarray*}
    |(dv_n^\pm(r,t))(\del_r)|&=&|dv_n\circ d\phi_n^\pm(r,t)(\del_r)|\\
    &\le&C[ 1-2(1-R_n^{-\mu})(1\mp r)]^2\,
    \frac2\mu  \frac{1-R_n^{-\mu}}{1-2(1-R_n^{-\mu})(1\mp r)}\\
    &=&\frac{2C}{\mu}[1-2(1-R_n^{-\mu})(1\mp r)]
  \end{eqnarray*}
  which is also uniformly bounded and converges to zero for
  $r\to\pm\frac12$ and $n\to\infty$.  So there exists a
  reparametrization and a subsequence so that $v$ converges to an
  $\H$--holomorphic map from the smooth part $\Sigma_0$ of the
  neck--nodal domain $(\hat\Sigma,j_0)$. Note that since the average twist on each neck
  converges to zero uniformly we have that $\int_{\{s\}times
    S^1}v^\ast \alpha\circ j\to 0$ uniformly, so in the limit
  condition \ref{eq:H_per} of Definition \ref{def:H_hol} is also
  satisfied along the necks which are now punctures for $\Sigma_0$.

  Similarly, we compute
  \begin{eqnarray*}
    |d\tilde v_n^N(r,t)-(T\,dt+S\,ds)(\del_t)|&\le& \frac{C}{R_n}\\
    |d\tilde v_n^N(r,t)-(T\,dt+S\,ds)(\del_r)|
    &\le& C\frac{R_n}{2(R_n-\ln(R_n))}+|S-2(R_n-\ln(R_n))\tilde S|\\
    &\le&\frac{C}2\left(1+\frac{\ln(R_n)}{R_n-\ln(R_n)}\right)
    +|S|\,\left|1-\frac{R_n-\ln(R_n)}{R_n}\right|\\
    &\le&\left(\frac{C}2+|S|\right)\frac{\ln(R_n)}{R_n}
  \end{eqnarray*}
  which converges to zero uniformly, so a subsequence of $\tilde
  v_n^N$ converges uniformly to a neck map.
  
  Using the diffeomorphisms $\phi_n$ we reinterpret our sequence of
  maps as maps from a fixed reference surface by gluing the thick part
  to the cylinders $C_1$. After passing to a subsequence the resulting
  domains with their induced complex structure converge to a
  neck--nodal domain, and the resulting maps converge uniformly in
  $C^0$ to a minimal twist neck--nodal $\H$--holomorphic map $v_0$. 

  Standard arguments show that the canonical lifts to the
  symplectization also converge, so $v_0$ is exact.
\end{proof}
In fact, it is not hard to extend these results to make the
convergence piecewise smooth, so that the convergence is $C^\infty$ on
the neck regions $\N$ and the smooth part $\Sigma_0$ of the domain.

\section{$S^1$--Invariant Stable Hamiltonian Manifolds}
\label{sec:circle_invariant}
In this section we consider circle--invariant stable Hamiltonian
manifolds of any (odd) dimension. We give topological conditions under
which the periods of families of $\H$--holomorphic maps are always
uniformly bounded, and thus obtain compact moduli spaces of
maps. $\H$--holomorphic maps into circle--invariant manifolds are
needed for applications to folded holomorphic maps
\cite{folded_holomorphic}. They also allow for the explicit
construction of examples highlighting the features of the compactness
theorem, which we give at the end of this section, as well as the
counterexample to a general compactness theorem which is constructed
in Section \ref{sec:non--compactness}.

\begin{definition}\label{def:circle_invariant}
  An stable Hamiltonian manifold $(Z,\alpha,\omega,J)$ is called
  $S^1$--invariant if the characteristic flow defines a free
  $S^1$--action that preserves $J$.
\end{definition} 

Any circle--invariant manifold is an $S^1$--bundle over a symplectic
manifold $(V,\omega_V)$ with projection $\pi_V:Z\to V$ so that
$\omega=\pi_V^\ast\omega_V$. The almost complex structure $J$ descends
to an $\omega_V$--compatible almost complex structure $J_V$ on $V$, so
$J=\pi_V|_F^\ast J_V$. For simplicity we will assume that
$d\alpha=C\cdot\omega$, where $C=\langle c_1(Z),[V]\rangle/vol(V)$,
where $c_1(Z)$ is the first Chern class of the bundle and $vol(V)$ is
the volume with respect to $\omega_V$. We can always arrange for
$\omega$ to be of this form. For details see Section 1 in
\cite{folded_holomorphic}.

There is a natural action of $Map(\hat\Sigma,S^1)$ on the space of
smooth maps into an $S^1$--invariant stable Hamiltonian manifold $Z$
given by
\begin{eqnarray*} 
  Map(\hat\Sigma,S^1)\times Map(\hat\Sigma,Z)\to
  Map(\hat\Sigma,Z) \qquad (f,v)\mapsto f\ast v
\end{eqnarray*} where
\begin{eqnarray*} 
  (f\ast v)(z)=\phi_{f(z)}(v(z))
\end{eqnarray*} 
and $\phi_t$ is the time--$t$ flow of the characteristic vector field
giving the circle action on $Z$. Since the circle action preserves
$J$, the action of of $Map(\hat\Sigma,S^1)$ leaves Equation
(\ref{eq:H_F}) invariant. Moreover, $(f\ast
v)^\ast\alpha=df+v^\ast\alpha$, so Equation (\ref{eq:H_L}) is
invariant under the action by $f\in Map(\hat\Sigma,S^1)$ if and only
if $f$ is harmonic on $\Sigma$. Let $H\subset Map(\hat\Sigma,S^1)$
denote the space of harmonic circle--valued functions on $\Sigma$. $H$
carries a circle action by adding a constant, and
$H/S^1=H^1(\Sigma;\setZ)$. This space is not compact and the counter
example Theorem \ref{thm:non-compactness} builds on this. The
intuition behind the topological conditions in Theorem
\ref{thm:compactness_S^1} is that they allow us to conclude that the
action of $H/S^1$ is free on the relative homotopy classes of
$\H$--holomorphic maps.

\begin{theorem}\label{thm:compactness_S^1}
  Let $u_n$ be a sequence of $\H$--holomorphic maps into an
  $S^1$--invariant almost contact manifold asymptotic to the same
  closed characteristics at the punctures and in the same relative
  homotopy class. Let $\pi:Z\to V$ be the projection to the
  base of $Z$. Assume that  one of the following holds:
  \begin{enumerate}[(i)]
  	\item\label{it:pi2} $\pi_2(V)$ is trivial.
  	\item\label{it:torus} $V=S^2$ and $\dot\Sigma$ is the
          once--punctured torus and the image of no $u_n$ 
          intersects the limit cycle and they are homotopic through
          maps that do not intersect the limit cycle.
  	\item\label{it:trivial} The bundle $Z$ is trivial.
  \end{enumerate}
  Then $u_n$ has a subsequence converging to a neck--nodal map.
\end{theorem}

\begin{proof}
  By Theorem \ref{thm:compactness} it suffices to show that the period
  integrals $P_{[\gamma]}(u_n)$ are uniformly bounded for any
  non--separating simple closed curve $\gamma$. Let $\phi_n$ be any
  sequence of Strebel differentials associated with the free homotopy
  class $[\gamma]$ and $(\Sigma,j_n)$, normalized to have sup norm 1
  (in the cylindrical metric). Choose a subsequence so that $\phi_n$
  converge with all derivatives to $\phi_0$. In particular the
  associated horizontal foliations converge. W.l.o.g. assume by
  possibly choosing another subsequence that $\gamma$ is a closed leaf
  of $\phi_1$ and that each foliation associated with $\phi_n$ has a
  closed leaf $\gamma_n$ that is close to $\gamma$ in the sense that
  \begin{eqnarray*}
    \left|\int_{\gamma}u_n^\ast\alpha-\int_{\gamma_n}u_n^\ast\alpha\right|<1.
  \end{eqnarray*}
  
  Thus it suffices to prove that
  \begin{eqnarray*}
    P_n=\int_{\gamma}u_n^\ast\alpha
  \end{eqnarray*}
  is uniformly bounded. We prove this by contradiction. Assume this
  was not the case, so there exists a subsequence that $|P_n|>n$.
  
  Let $\tilde u_n=\pi\circ u_n:\Sigma\to V$ denote the
  projection of the maps into $V$, which extend naturally over the
  punctures. These maps are $J_V$--holomorphic, and by the usual
  Gromov compactness we may choose a further subsequence so that
  $\tilde u_n$ converge to a map $\tilde u_0:\Sigma\to V$. For $N$
  large enough and any $n,m\ge N$ there exists a unique vector field
  $\xi_{n,m}\in \tilde u_n^\ast TV$ so that
  \begin{eqnarray*}
    \tilde u_m(z)=\exp_{\tilde u_n(z)}(\xi_{n,m}(z))
  \end{eqnarray*}
  satisfying $\norm{\xi_{n,m}}_\infty\to 0$ and $\int_\Sigma
  \norm{\xi_{n,m}}^2\to 0$ as $n,m\to 0$.
  
  Let $H_{n,m}:[0,1]\to Z$ be a relative homotopy between
  $u_n$ and $u_m$ and consider the ``flux'' given by the difference in
  period integrals
  \begin{eqnarray*}
    \F(v_m,v_n)(\gamma)=P_m-P_n=\int_{\gamma}u_m^\ast\alpha-u_n^\ast\alpha
    =\int_{[0,1]\times \gamma}H^\ast d\alpha
    =\int_{[0,1]\times \gamma}(\pi\circ H)^\ast c_v\omega_V
  \end{eqnarray*}
  and let $G_{n,m}:[0,1]\times \dot\Sigma\to V$ be the homotopy,
  relative to the punctures, between $\tilde u_n$ and $\tilde u_m$
  given by
  \begin{eqnarray*}
    G_{n,m}(t,z)=\exp_{\tilde u_m(z)}(t\xi_{m,n}(z)).
  \end{eqnarray*}
  Then the composition $K=(\pi\circ H)\ast G$ is a homotopy between
  $\tilde u_n$ and $\tilde u_n$, and we choose another subsequence
  by dropping finitely many terms so that $\left|\int_{\Sigma}G^\ast
    \omega_V\right|<1$, and thus
  \begin{eqnarray}\label{eq:homotopy_periods}
    |P_n|\le|P_1|+|P_n-P_1|
    \le |P_1|+1+\left|\int_{S^1\times \gamma}K^\ast c_v\omega_V\right|.
  \end{eqnarray}
  We finish the proof by showing that the integral in Equation
  (\ref{eq:homotopy_periods}) vanishes and thus the periods are
  uniformly bounded.

  In the case (\ref{it:pi2}) consider
  the fibration
  \begin{eqnarray*}
    X=Map\left((\hat\Sigma,\del\hat\Sigma),(V,\{p_i\})\right) \to
    Map(\hat\Sigma,V)\to Map(\del\hat\Sigma,\{p_i\}).
  \end{eqnarray*} 
  This gives rise to the long exact sequence in homotopy
  \begin{eqnarray*}
    0\to\pi_1(X)\to\pi_1(Map(\hat\Sigma,V))=\pi_1(Map(\bigvee
    S^1,V))=\bigoplus\pi_2(V)\to 0.
  \end{eqnarray*} 
  Thus $\pi_1(X)=\bigoplus \pi_2(V)$ so in the case (\ref{it:pi2}) the
  loop of maps $K$ is contractible and the integral in Equation
  (\ref{eq:homotopy_periods}) vanishes.

  In the case (\ref{it:torus}) note that $\omega_V=c\cdot PD[pt]$. For
  the once--punctured torus we may choose the Strebel cylinders so
  that the puncture is on the boundary of the cylinder. To compute the
  integral in Equation (\ref{eq:homotopy_periods}) we take the curve
  $\gamma$ to be the central curve in the Strebel cylinder. For a
  family of degree $d$ maps mapping the puncture $p$ to the closed
  characteristic over $\infty\subset S^2$ and not intersecting the
  limit cycle so that $u^{-1}(\infty)=p$ we have
  \begin{eqnarray*}
    K^\ast \omega_V=c\,K^\ast PD[pt]=c
    PD[u^{-1}(\infty)]=c\,PD[S^1\times d\cdot \{p\}]
  \end{eqnarray*}
  and thus
  \begin{eqnarray*}
    \int_{S^1\times \gamma_i}K^\ast c_v\omega_V
    =c\,c_v\,\#(S^1\times\gamma_i,S^1\times d[p])
    =c\,c_v\,\#(\gamma,d[p])=0
  \end{eqnarray*}
  since the curve $\gamma$ does not intersect $p$ by construction.

  In the case (\ref{it:trivial}) the integral vanishes since
  $c_v=0$.
\end{proof}
The obvious case absent from the theorem is when $V=S^2$ and we
consider more general maps than case (\ref{it:torus}). This will be
addressed in Section \ref{sec:circle_invariant} and it turns out that in
general the space of maps is not compact in this case. The key to
proving compactness of the space of maps was to show that the ``flux''
integrals are uniformly bounded. This can either be achieved by
topological restrictions on the target as in cases (\ref{it:pi2}) or
(\ref{it:trivial}) or by assumptions on the space of domains as in
(\ref{it:torus}).

At this point it is convenient to give some explicit examples to
illuminate aspects of the compactness theorem. 

First we give an example of a sequence of maps converging to a
neck--nodal map where the neck converges to a characteristic flow line.
\begin{example}\label{ex:string}
  Consider $Z=S^1\times S^2$, the trivial bundle with an
  $S^1$--invariant structure. Maps $v:\dot\Sigma\to Z$ are
  $\H$--holomorphic if and only if the projection $\tilde v=\pi_{S^2}
  v$ is $J$--holomorphic and the projection $\theta=\pi_{S^1}v$ is
  harmonic. Let $(\tilde v_t,j_t):T^2\to S^2$ be a family of
  $J$--holomorphic maps so that the domains converge to a
  once--pinched torus, pinched along a simple closed curve in the
  homology class $A\in H_1(T^2;\setZ)$ with a dual class $B$. Then
  $(v_t,j_t):S^2\to Z$ $v_t=(\theta=0,\tilde v_t)$ is
  $\H$--holomorphic (and in fact also $J$--holomorphic). Mark a point
  $p$ in $T^2$ away from the node.

  Consider the family of $j_t$ harmonic circle--valued functions
  $\theta_t:T^2\to S^1$ with $\theta_t(p)=0$ and periods 0 along $A$
  and 1 along $B$. Set $u_t=\theta_t\ast v_t=(\theta_t,\tilde
  v_t)$. Then $u_t$ is $\H$--holomorphic and the image neck domain
  converges to a characteristic flow line of length 1 over the image
  of the node $q\in S^2$. More precisely, the neck map converges to a
  map $x:[-1/2,1/2]\times S^1\to Z$, $x(s,t)=(s,q)$.
\end{example}
\begin{figure}[htbp]
  \centering
  \includegraphics[width=9cm]{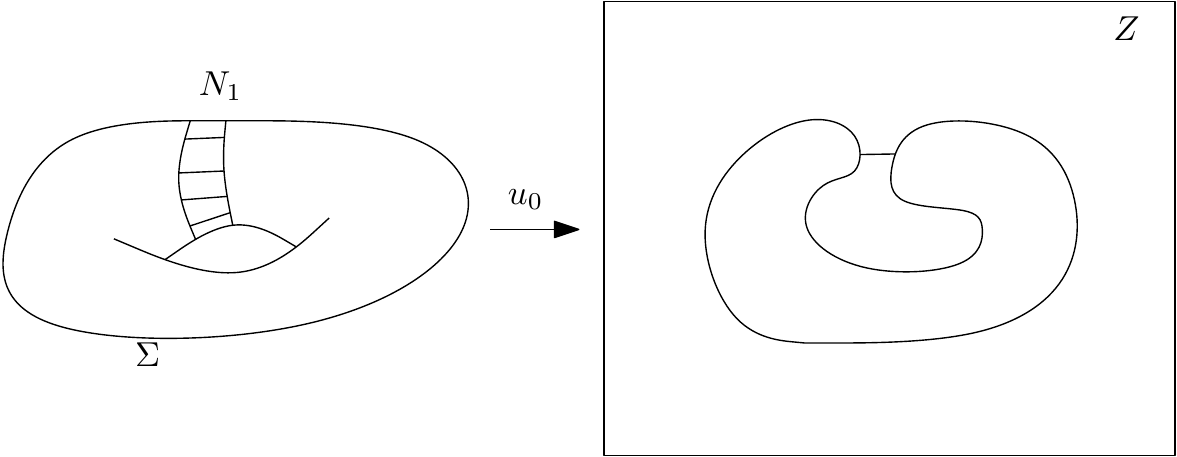}
  \caption{Map with twist at node.}
  \label{fig:string}
\end{figure}

Next we show that ``breaking of trajectories'' can happen in a compact
subset in the symplectization. 
\begin{example}\label{eq:broken_trajectory}
  With the same notation as in the previous example, consider a family
  of $j_t$ harmonic circle--valued functions $\theta_t:T^2\to S^1$
  with $\theta_t(p)=0$ and periods 1 along $A$ and 0 along $B$. Set
  $u_t=\theta_t\ast v_t=(\theta_t,\tilde v_t)$. Then $u_t$ is
  $\H$--holomorphic and the neck domain converges to a cylinder over a
  closed characteristic over the image of the node $q\in S^2$. More
  precisely, the neck map converges to a map $x:[-1/2,1/2]\times
  S^1\to Z$, $x(s,t)=x_q(t)=(t,q)$, where $x_q$ is the parametrized
  closed characteristic over $q$. The lift to the symplectization sits
  in a constant $\setR$--slice.
\end{example}
\begin{figure}[htbp]
  \centering
  \includegraphics[width=9cm]{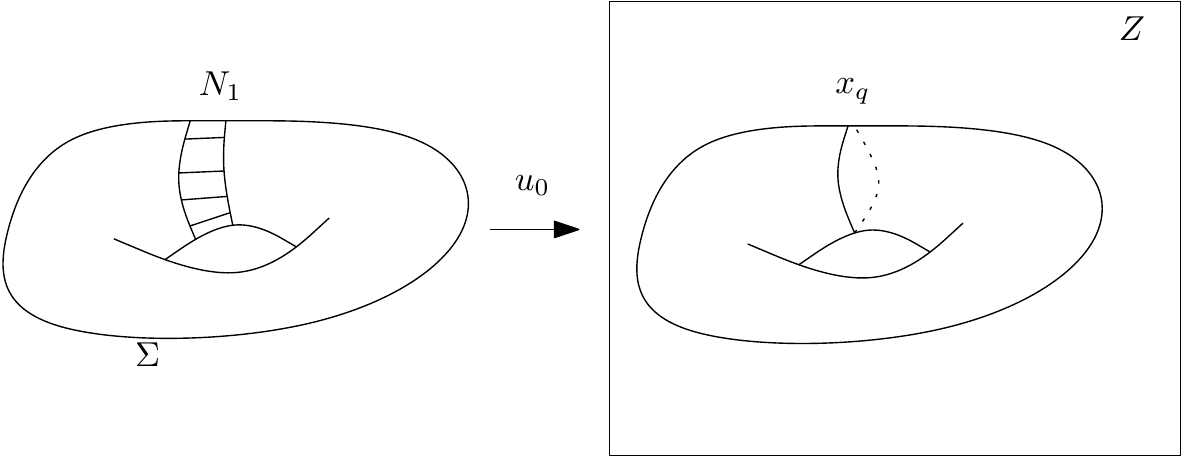}
  \caption{Map broken at closed characteristic in compact subset of symplectization.}
  \label{fig:broken_trajectory}
\end{figure}

To see how the minimality of the twist comes into the compactness
statement, consider the following example.
\begin{example}\label{eq:minimal_twist}
  Let $\tilde v:T^2\to S^2$ be a $J$--holomorphic map with two homologous
  neck regions $N_1$ and $N_2$ that is pinched along the two
  necks and lift this to a map $v:T^2\to Z=S^1\times S^2$ via
  $v=(0,\tilde v)$. 

  Orient the necks so that $[N_1]+[N_2]=0\in H_1(\Sigma;\setZ)$. The
  periods of $v$ are zero, and the twists $S_1$ and $S_2$ are
  zero. However, let $\chi$ be the function on $\dot \Sigma$ that
  equals 1 on one component of the normalization of $\Sigma$ and 0 on
  the other. Then $v_t=(t\chi,\tilde v)$ is also an $\H$--holomorphic
  map with twists $S_1^t=S_2^t$. So the periods still vanish since
  $S_1^t-S_2^t=0$ for any $t\in \setR$. This gives a non--compact
  family of maps with bounded periods, but $v_t$ does not have minimal
  twist for $t\ne 0$, so these are not minimal twist neck--nodal maps
  and cannot arise as limits of smooth $\H$-holomorphic maps unless
  $t=0$.
\end{example}

We now consider an example where the lift to the symplectization
develops an ``infinite funnel'', i.e. a node is forming on one neck
region (zero period), but the node is pushed off to infinity in the
symplectization. These following two examples motivate why we have
opted to focus our discussion on $\H$--holomorphic maps into the contact
manifold instead of their lifts to the symplectization.
\begin{example}\label{ex:funnel}
  Let $Z=S^3$ be the Hopf fibration over $S^2$ with the canonical
  contact from $\alpha$. Let $(v_t,j_t):\dot T^2\to Z$ be a family of
  $J$--holomorphic maps from the once--punctured torus so that the
  lifts to the symplectization $\hat v_t=(a_t,v_t)$ converge to a
  building of height two where the torus is pinched along two
  homologous neck regions $N_1$ and $N_2$, oriented as the boundary of
  the component of the smooth part of the limit domain that does not
  contain the puncture.  The necks converge to a cylinder over a
  closed characteristic of positive integer period $m_1$ and $m_2$
  on $N_1$ and $N_2$.

  Let $A=[N_1]=-[N_2]\in H_1(T^2;\setZ)$ and $B$ a dual class. Consider
  a family of $j_t$--harmonic circle--valued functions $f_t:T^2\to
  S^1$ with periods $-m_1$ on $A$ and 0 on $B$, and let $u_t=f_t\ast
  v_t$ be the corresponding family of $\H$--holomorphic maps that lift
  to the symplectization as $\hat u_t=(a_t,u_t)$. Then $a_t$ is
  unchanged from before, so the limit is a level 2 curve. But the neck
  $N_1$ now has period 0 (and $N_2$ has period $m_1+m_2>0$). So the neck
  $N_1$ converges to an ``infinite funnel'' in the symplectization. 
\end{example}

\begin{figure}[htbp]
  \centering
  \includegraphics[width=13cm]{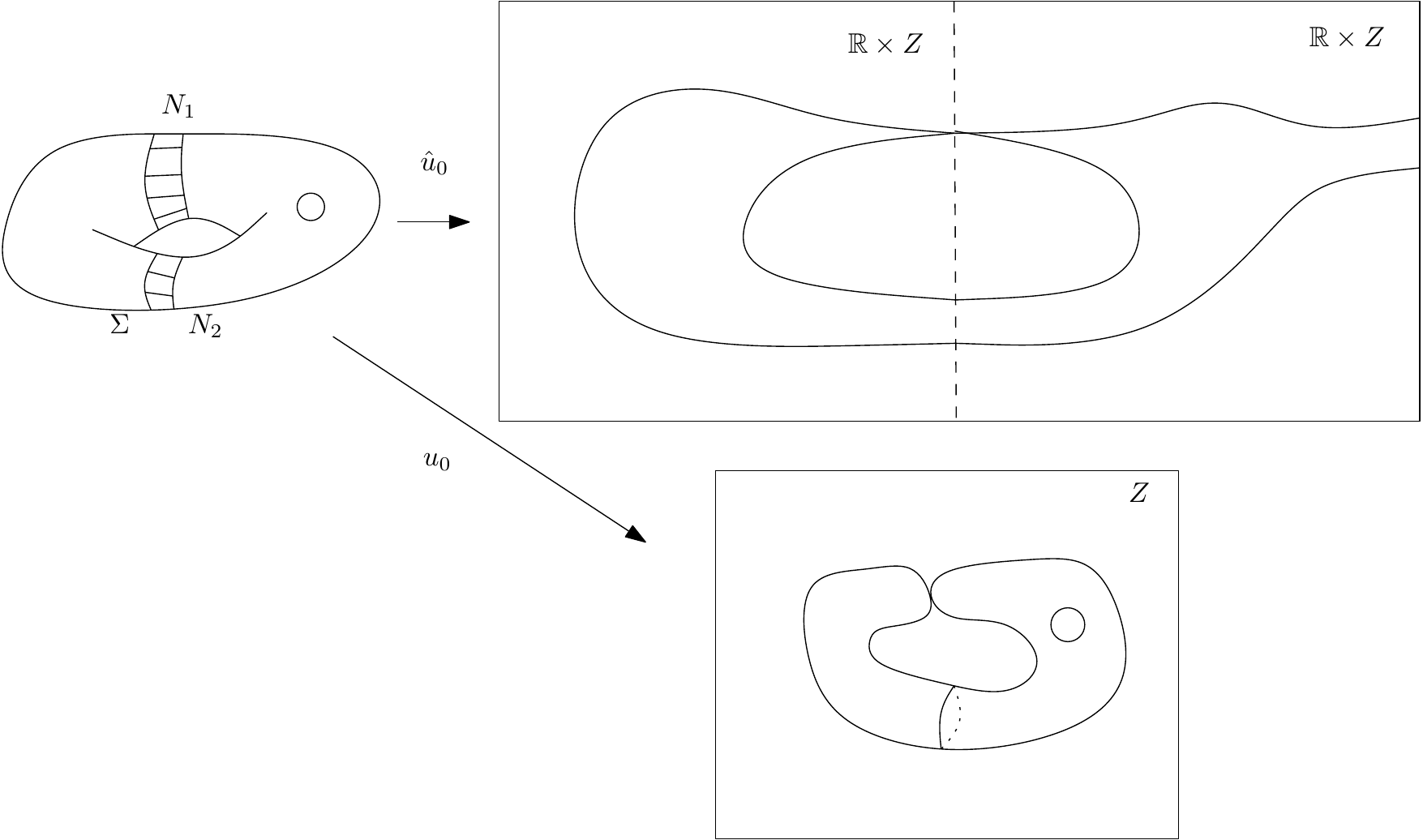}
  \caption{Map with funnel.}
  \label{fig:funnel}
\end{figure}

This example can easily be modified to include non--vanishing twist at
the funnel.
\begin{example}\label{ex:funnel_twist}
  With the same notation as in the previous example choose a function $f_t:T^2\to
  S^1$ with periods $-m_1$ on $A$ and 1 on $B$, and let $u_t=f_t\ast
  v_t$ be the corresponding family of $\H$--holomorphic maps. Then the
  limit $u_0$ has non--zero twist $S_1=\frac{m_1}{m_1+m_2}>0$  at the
  ``infinite funnel'', and twist $S_2=\frac{m_2}{m_1+m_2}$ at the neck
  $N_2$ converging to a closed characteristic. 
\end{example}

\begin{figure}[htbp]
  \centering
  \includegraphics[width=13cm]{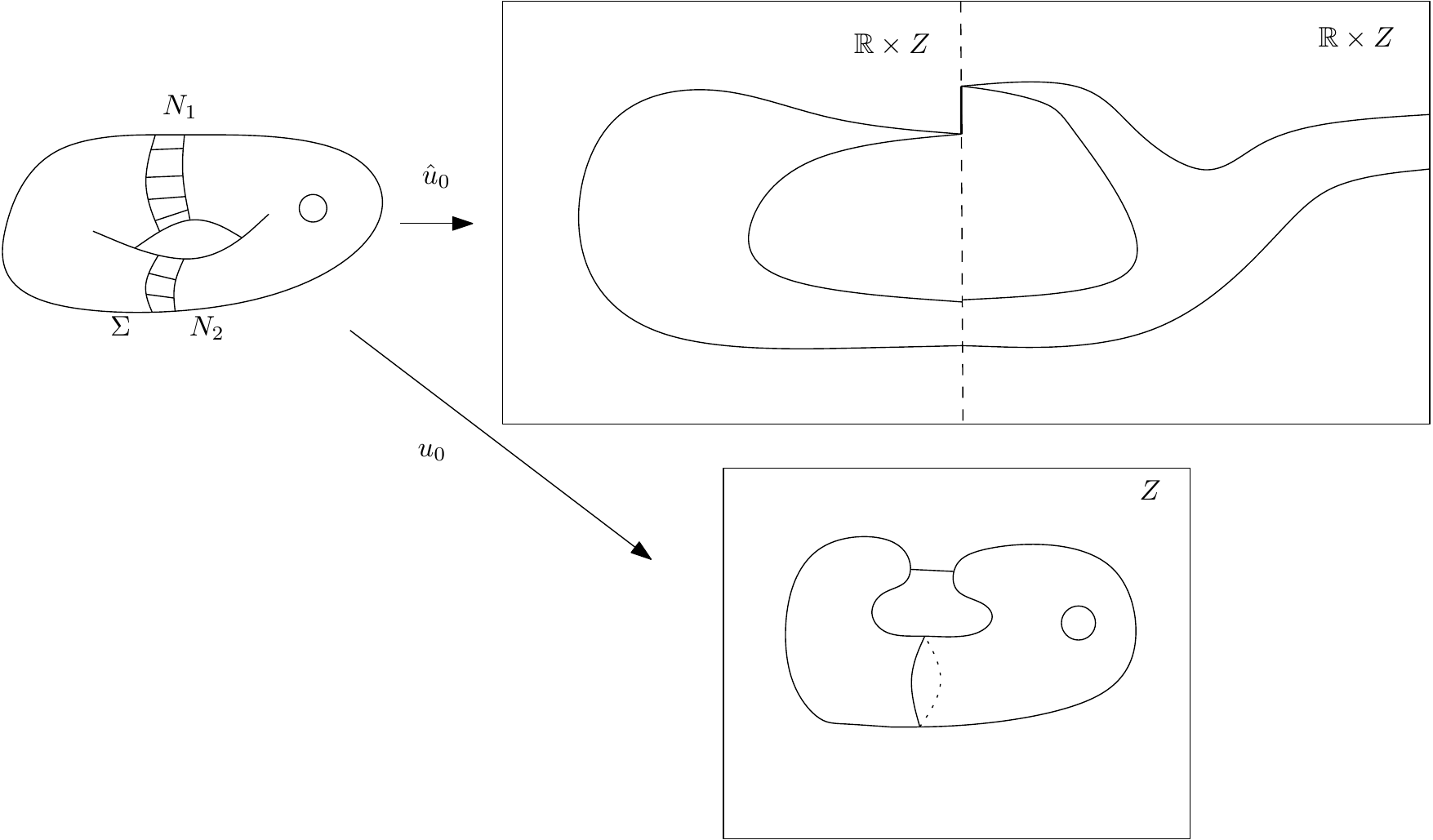}
  \caption{Map with twist at funnel.}
  \label{fig:funnel_twist}
\end{figure}

\section{Non--Compactness Results}
\label{sec:non--compactness} In this section we prove Theorem
\ref{thm:non-compactness}.  Let $\pi:Z\to S^2$ be a principle
$S^1$--bundle with connection 1--form $\alpha$ and curvature form
$\omega=d\alpha>0$. Choose an $S^1$--invariant almost complex
structure $J$, so $(Z,\alpha,\omega,J)$ is a circle--invariant stable
Hamiltonian structure (see Definition \ref{def:circle_invariant}) with
bundle projection $\pi:Z\to V=S^2$ and first Chern class $c_1(Z)\ge
0$.

Consider $\H$--holomorphic maps $v$ from a genus $g$ Riemann surface
$(\Sigma_g,j)$ with punctures $p_1,\ldots p_l\subset\Sigma_g$ with
asymptotics on a closed characteristic $C=\pi_V^{-1}(\infty)$ of
multiplicity $m_1,\ldots m_l$, respectively, where $m_i\ne 0$ for
$i=1,\ldots,l$ and $\sum_{i=1}^l m_i=d\cdot c_1(V)$.

In this situation the action of $H^1(\Sigma,\setZ)$ described in
Section \ref{sec:circle_invariant} on
$Map\left((\hat\Sigma,\del\hat\Sigma),(Z,C)\right)$ is trivial in
homotopy, and we will show that we can moreover find a smooth path of
$\H$--holomorphic maps connecting an $\H$--holomorphic maps $v$ to
$f\ast v$, where $f$ is a non--trivial element in $H^1(\Sigma;\setZ)$.

Consider the projection $u=\pi_V v$ of such maps to the base
$S^2$. These are degree $d$ holomorphic maps with the punctures $p_i$
mapping to $\infty\in S^2$. So $u$ is a rational function with poles
at $p_i$ (with points repeated according to the multiplicity
$m_i$). Up to $C^\ast$ action on $S^2$, degree $d$ rational functions
with these prescribed poles are in 1--1 correspondence with divisors
$D=\sum_{i=1}^d(p_i-q_i)$ of degree $0$ on the domain
with vanishing abelian sum
\begin{eqnarray*}
\mu(D)=\left(\sum_{i=1}^d\int_{q_i}^{p_i}\omega_1,\ldots,
  \sum_{i=1}^d\int_{q_i}^{p_i}\omega_g\right)
\end{eqnarray*} where $\omega_1,\ldots\omega_g$ is a normalized basis
of holomorphic 1--forms on $(\Sigma_g,j)$ (see
e.g. \cite{griffiths_harris} p 235).

Let $D_t=m_1\,p_1(t)+\sum_{i=m_1+1}^dp_i-\sum_{i=1}^dq_i(t)$, $t\in
\setR/\setZ$, be a smooth family of divisors on $\Sigma_g$ with
$\mu(D_t)=0$ so that the loop $\{p_1(t)\}\subset \Sigma$ is
non--trivial in homology, i.e. there exists a loop $\gamma$ so that
$\#(\gamma,\{p_1(t)\})=1$. 

To give an explicit example, consider the flat torus
$\setR/\setZ\times \setR/\setZ$ with basis of holomorphic 1--forms
$\omega_1=dz=dx+i dy$ and let $p_1(t)=(t,0)$, $q_1(t)=(t+\frac12,0)$,
$p_2=(\frac12,\frac12)$, $q_2=(0,\frac12)$. Then, with
$D_t=p_1(t)-q_1(t)+q_2-q_2$,
\begin{eqnarray*}
  \mu(D_t)&=&\int_{q_1(t)}^{p_1(t)}dz+\int_{q_2(t)}^{p_2(t)}dz
  =\int_{q_1(t)}^{p_1(t)}dx+\int_{q_2(t)}^{p_2(t)}dx+i\left(
    \int_{q_1(t)}^{p_1(t)}dy+\int_{q_2(t)}^{p_2(t)}dy\right)\\
  &=&-\frac12+\frac12+0i=0
\end{eqnarray*}
and $\{p_1(t)\}$ intersects the loop $\{y=0\}$ once.

Now consider a corresponding loop of holomorphic maps $\tilde
u:\setR/\setZ\times \Sigma_g\to V$, where $u_t=\tilde
u(t,\cdot)$ is a rational function corresponding to $D_t$. We are
interested in the flux
\begin{eqnarray*} \F(u_0,u_1)[\gamma]
&=&\int_{S^1\times\gamma}u^\ast\omega
=c\,\int_{S^1\times\gamma}u^\ast PD[pt]
=c\,\#(S^1\times\gamma,[u^{-1}(\infty)])\\ 
&=&c\,\sum_{i=1}^l m_i\#(S^1\times\gamma,[\{t,p_i(t)\}])
=c\,m_1\#(\gamma, \{p_i(t)\}_{t\in S^1})\\
&=&c\,m_1\ne0.
\end{eqnarray*}

Lift the family of maps $\tilde u$ to a family of $\H$--holomorphic
maps
\begin{eqnarray*} 
  \tilde v:\setR\times \dot\Sigma_g\to Z
\end{eqnarray*} 
with the prescribed asymptotics at the punctures $p_i$. Choose a
sequence $t_i\in\setR$ and consider the sequence $v_t=\tilde
v(t,\cdot)$ of $\H$--holomorphic maps in the 1--parameter family
$\tilde v$. Then for the loop $\gamma$ in $\dot\Sigma$ the flux
\begin{eqnarray*}
  \F(v_1,v_i)(\gamma)=\int_{\gamma}(v_1^\ast\alpha-v_i^\ast\alpha)
\end{eqnarray*}
is of the order of $m_1(t_i-t_1)$, more precisely it lies in the
interval $[m_1(t_i-t_1)-M,m_1(t_i-t_1)+M]$ where $M=\sup_{t\in
  [0,1]}\F(v(0,\cdot),v(t,\cdot))(\gamma)$. Thus there exists a
subsequence with bounded flux if and only if there exists a
subsequence with bounded $t_i$, i.e. if there exists a convergent
subsequence of $t_i\in\setR$.

In the above example the harmonic 1--form becomes unbounded in the sup
norm w.r.t. the cylindrical metric. But this is not the essential
condition that keeps the space of $\H$--holomorphic maps in a fixed
relative homotopy class from being compactified as the following
example shows, where the harmonic 1--form $|\eta|$ is bounded by
$|da|$ in the cylindrical metric. 

Consider maps from the once--punctured torus to $S^3$ with the standard
structures and let $v_n$ be a sequence of $J$--holomorphic maps so
that the conformal structure is biholomorphic to 
\begin{eqnarray*}
  S^1\times n\cdot S^1
  =\{(e^{2\pi i\,s},e^{2\pi i\,t/n})\},
\end{eqnarray*}
so the complex structure $j\del_s=\del_t$ degenerates, and assume that
the sequence converges to a $J$--holomorphic map $v_\infty$ from a
bubble domain given by two spheres joined at two points, i.e. the
complex structure pinched along two circles in the class of $S^1\times
\{pt\}$. Further assume that the two nodes are non--trivial, i.e. the
wrap a closed characteristic of period $\tau\ne 0$. In particular,
$v^\ast\alpha\approx\tau\,ds$ on the necks. The cylindrical metric is
given by $g_n=ds^2+ dt^2$.

We construct new maps $\hat v_n=f_n\ast v_n$, where
$f_n(s,t)=e^{2\pi i\,t}$ where $(s,t)\in S^1\times n\cdot S^1$. Then
$df_n=dt$, so it is bounded in the cylindrical metric (and also bounded
by $da=v^\ast\alpha$ if $\tau$ is large enough). 

\appendix

\section{Harmonic 1--Forms}
\label{sec:harmonic-1-forms}
Here we establish some properties of harmonic 1--forms that we use.
\begin{lemma}\label{lem:eta_on_necks}
  Any harmonic 1--form $\eta$ on $C_R=[-R,R]\times S^1$ satisfies
  \begin{eqnarray*}
    |\eta(s,t)-(\tilde S\,ds+T\,dt)|
    \le \rho(s)\norm{\eta-(\tilde S\,ds+T\,dt)}_{\del C_R,\infty} .
  \end{eqnarray*}
  where $T$ is the period and $\tilde S$ is the average twist.
\end{lemma}
\begin{proof}
  Any harmonic 1--form $\eta$ on $C_R$ is of the form
  \begin{eqnarray*}
    \eta(s,t)=(\tilde S\,ds+T\,dt)+f(s,t)ds+g(s,t)dt
  \end{eqnarray*}
  where $f$ and $g$ are harmonic function with vanishing
  average. Expanding $f$ and $g$ into Fourier series in the
  $t$--variable shows that $f$ and $g$ are sums of terms consisting of
  products of $\sinh(ns)$ and $\cosh(ns)$ with $\sin(nt)$ and
  $\cos(nt)$ plus a linear term in $s$. Remembering that
  $d\eta=d(\eta\circ j)=0$ we see that the term that is linear in $s$
  must in fact vanish. Since $f\,ds+g\,dt$ have vanishing average
  twist and center action, only terms with $n>0$ appear in the
  sums. Now note that
  \begin{eqnarray*}
    \cosh(ns)\le 2\cosh(s)\cosh((n-1)s),\qquad\mathrm{and}\quad
    |\sinh(ns)|\le 2\cosh(s)|\sinh((n-1)s)|
  \end{eqnarray*}
  and thus
  \begin{eqnarray*}
    \frac{\cosh(ns)}{\cosh(nR)}\le e^{-R}\cosh(s)\le \rho(s),
    \qquad
    \frac{\cosh(ns)}{\cosh(nR)}\le e^{-R}\cosh(s)\le \rho(s),
  \end{eqnarray*}
  remembering that $\rho^2(s)=8 e^{-2R}\cosh(2s)$ and thus
  $e^{-R}\cosh(s)\le \rho(s)$.

  Using the Fourier expansions of $f$ we obtain
  \begin{eqnarray*}
    |f(s,t)|\le \rho(s)\norm{f(s,t)}_{\del C_R,\infty} 
  \end{eqnarray*}
  with an analogous estimate holding for $g$ which gives the desired result.
\end{proof}

\begin{lemma}\label{lem:bounded_twist}
  Fix a basis $\{\gamma_i\}_{i=1,\ldots 2g}$ of $H_1(\Sigma;\setZ)$
  and let $C_R=[-R,R]\times S^1$ with $R>4$ be a neck cylinder in
  $\Sigma$. 

  There exists a constant $c>0$ so that for any $\eta$ be a harmonic
  1--form on $\Sigma$ with periods on $\{\gamma_i\}$ bounded by $C$ the
  twist $S$ of $\eta$ on $C_R$ satisfies $S_n<c\,C$.
\end{lemma}
\begin{proof}
  Using a change of symplectic basis we may assume that 
  $\gamma_1$ is in the class of $C_R$ and $\gamma=\gamma_{g+1}$ is its
  dual basis element, and that all periods with respect to this new
  basis are bounded by $C\,c/2$. Assume without loss of generality that
  $\eta$ has vanishing periods on all of the these basis elements
  except $\gamma$. 

  Let $\beta:[-R,R]\to [0,1]$ be a bump function supported on
  the interior of the interval that equals 1 at all points with
  distance at least 1 from the boundary with integral $\tilde C$. Note
  that $2(R_n-2)\le\tilde C\le 2R_n$. Now consider the closed 1--form
  $\nu(s,t)=\frac{C}{\tilde C}\beta(s)ds$ on $C_R$. It has the same
  periods as $\eta$, and $L^2$ norm
  \begin{eqnarray*}
    \langle\nu,\nu\rangle
    \le\frac{C^2}{\tilde C^2}2R\le C^2\frac{R}{(R-2)^2}\le 2\frac{C^2}{R}.
  \end{eqnarray*}
  On the other hand $\eta=\tilde S\,ds+\mu$, where $\mu$ has average 0
  on $C_R$. Then
  \begin{eqnarray*}
    \langle\eta,\eta\rangle=\langle\mu,\mu\rangle+ \tilde S^22R\ge\frac {S^2}{2R}.
  \end{eqnarray*}
  Since $\eta$ minimizes the $L^2$ norm among all closed forms with
  the same periods, we conclude that the twist $S<c\,C$.
\end{proof}

\begin{lemma}\label{lem:harmonic_bounded}
  Let $\{\gamma_i\}$ be a basis of $H_1(\Sigma;\setZ)$. 
  There exists a constant $C>0$, independent on the complex structure
  $j$ on $\Sigma$ so that for any harmonic 1--form
  $\eta$ on $(\Sigma,j)$
  \begin{eqnarray*}
    \norm{\eta}_\infty\le C\norm{P(\eta)}
  \end{eqnarray*}
  where $P(\eta)$ is the period map w.r.t. $\{\gamma_i\}$ and the
  $\norm{.}_\infty$ is taken with respect to the cylindrical metric.
\end{lemma}
\begin{proof}
  Since both sides are linear under scaling, it suffices to prove the
  statement for harmonic 1--forms with $P(\eta)=1$.  If this was not
  true, there exists a sequence of surfaces $(\Sigma,j_n)$ and
  harmonic 1--forms $\eta_n$ so that
  \begin{eqnarray*}
    M_n=\norm{\eta_n}_\infty >n\norm{P(\eta_n)}=n. 
  \end{eqnarray*}
  Consider the rescaled harmonic 1--forms $\mu_n=\eta_n/M_n$ with
  periods less than $\frac1n$. By elliptic regularity and Arzela
  Ascoli we can extract a subsequence that converges uniformly in
  $C^\infty$ to a the thick part of $(\Sigma,j_n)$. Using Lemma
  \ref{lem:eta_on_necks} we see that $\eta_n$ then also converge
  uniformly in $C^\infty$ on the necks, so $\eta_n$ converge uniformly
  in $C^\infty$ to a harmonic 1--form $\eta$ with vanishing periods on
  the limit surface with cylindrical ends. The periods of $\eta_n$ on
  the neck regions converge to 0 uniformly, and the twist converge to
  zero by Lemma \ref{lem:bounded_twist}. By Lemma
  \ref{lem:eta_on_necks} we then see that the limit $\eta$ still has
  sup norm 1. The cylindrical ends are conformally equivalent to
  punctured disks, and $\eta$ extends to a harmonic 1--form over the
  punctures to a 1--form on the normalization of the closed nodal
  surface with with vanishing periods, again using Lemma
  \ref{lem:eta_on_necks}. By the Hodge Theorem $\eta$ vanishes
  identically on the normalization, contradicting that the sup norm of
  $\eta$ is 1.
\end{proof}


\newcommand{\etalchar}[1]{$^{#1}$}
\providecommand{\bysame}{\leavevmode\hbox to3em{\hrulefill}\thinspace}
\providecommand{\MR}{\relax\ifhmode\unskip\space\fi MR }
\providecommand{\MRhref}[2]{%
  \href{http://www.ams.org/mathscinet-getitem?mr=#1}{#2}
}
\providecommand{\href}[2]{#2}

\end{document}